\documentclass[a4paper,leqno,12pt]{amsart}
\usepackage[all]{xy}
\usepackage{xspace}
\usepackage{bm}
\usepackage{amsmath}
\usepackage{amstext}
\usepackage{amsfonts}
\usepackage[mathscr]{euscript}
\usepackage{amscd}
\usepackage{latexsym}
\usepackage{amssymb}
\usepackage{mathrsfs}
\usepackage{cancel}
\usepackage{graphicx}
\usepackage{enumerate}
\usepackage{color}

\theoremstyle{plain}
    \newtheorem{thm}{Theorem}[section]

    \newtheorem{lemma}{Lemma}
    
    \newtheorem {seq}{Consequence}

    \newtheorem{subsec}[thm]{}

\theoremstyle{definition}

\newenvironment{myeq}[1][]
{\stepcounter{thm}\begin{equation}\tag{\thethm}{#1}}
{\end{equation}}

\newcommand{\mysdiag}[2][]
{\stepcounter{thm}\begin{equation}
     \tag{\thethm}{#1}\vcenter{\xymatrix@R=25pt@C=-25pt{#2}}\end{equation}}
\newcommand{\mytdiag}[2][]
{\stepcounter{thm}\begin{equation}
     \tag{\thethm}{#1}\vcenter{\xymatrix@R=25pt@C=15pt{#2}}\end{equation}}
\newcommand{\myudiag}[2][]
{\stepcounter{thm}\begin{equation}
     \tag{\thethm}{#1}\vcenter{\xymatrix@R=27pt@C=12pt{#2}}\end{equation}}
\newcommand{\myvdiag}[2][]
{\stepcounter{thm}\begin{equation}
     \tag{\thethm}{#1}\vcenter{\xymatrix@R=20pt@C=16pt{#2}}\end{equation}}
\newcommand{\mywdiag}[2][]
{\stepcounter{thm}\begin{equation}
     \tag{\thethm}{#1}\vcenter{\xymatrix@R=22pt@C=32pt{#2}}\end{equation}}
\newcommand{\myxdiag}[2][]
{\stepcounter{thm}\begin{equation}
    \tag{\thethm}{#1}\vcenter{\xymatrix@R=30pt@C=25
      pt{#2}}\end{equation}}
\newcommand{\myzdiag}[2][]
{\stepcounter{thm}\begin{equation}
     \tag{\thethm}{#1}\vcenter{\xymatrix@R=10pt@C=25pt{#2}}\end{equation}}
\newcommand{\hsm}{\hspace*{2 mm}}
\newcommand{\vsn}{\vspace{2 mm}}

\begin{document}
%
%
\title{An infinite branch in a decidable tree}
%
\author[Soprunov S.F.]{Soprunov S.F.}
\email{soprunov@mail.ru}

\begin{abstract}
%
%
We consider a structure $\mathcal {M} = \langle \mathbb N, \{Tr,<\} \rangle$, where the relation $Tr(a,x,y)$ with a parameter $ a$ defines a family of trees on $\mathbb N$ and $<$ is the usual order on $\mathbb N$. We show that if the elementary theory of $\mathcal M$ is decidable then (1) the relation $Q( a) \rightleftharpoons$ "\emph{there is an infinite branch in the tree} $Tr( a,x,y)$" is definable in $\mathcal M$, and (2) if there is an infinite branch in the tree $Tr( a,x,y)$, then there is a definable in $\mathcal M$ infinite branch.

\end{abstract}

\maketitle
\section{Preliminaries }
Let $Tr(x,y)$ be a tree on the $\mathbb N$, we are interested in whether there is an infinite branch in this tree. If the tree is locally finite then, according K\"{o}nig's lemma~\cite{b}, an infinite branch exists iff the tree is infinite. It is easy to notice, that in this case an infinite branch can be defined in the structure $\langle \mathbb N, \{Tr,<\} \rangle$. The question is more complicated for an arbitrary tree.

We show that if a family $Tr( a,x,y)$ of trees with a parameter $ a$ such, that the elementary theory of $\mathcal {M} = \langle \mathbb N, \{Tr,<\} \rangle$ is decidable then (1) the relation $Q( a) \rightleftharpoons$ "\emph{there is an infinite branch in the tree} $Tr( a,x,y)$" is definable in $\mathcal M$, and (2) if there is an infinite branch in the tree $Tr( a,x,y)$, then there is a definable in $\mathcal M$ infinite branch in the tree $Tr( a,x,y)$. 

For simplicity hereinafter we write $a$ instead of  $\bar a$ in parameters though all parameters could be vectors as well as numbers.

The proof consists of two steps. First we show, that if a tree is in some sense complicated, then the theory of the corresponding structure is undecidable. Second we show, that if a tree is not complicated, then (1) and (2) holds. To demonstrate undecidability we use an interpretation of fragments of the arithmetic in the structure~\cite{t}.

\section{Interpretation }

In this section we consider a structure $\mathcal {M}=\langle \mathbb N, \Sigma \rangle$, the usual order $<$ belongs to $\Sigma$. Suppose that subset $S \subset \mathbb N$ is finite and a relation $B(x, y)$ is definable in $\mathcal {M}$. By $s^B_i$ we denote $S \cap \{x|B(x,i)\}$ and say, that $B$ \emph{realises the number k on S} ($k \leqslant |S|$) if $\{s^B_i | i \in \mathbb N\}=\{s \subset S | |s|=k\}$. The property to realise a number can be expressed by the statement:
\begin{equation*}\label{f1}
(\forall i,j)(s^B_i \subset s^B_j \to s^B_i=s^B_j) \land (\forall i,a,b)(a \in s^B_i \land b \in S \setminus s^B_i \to (\exists j)(s^B_j = s^B_i \cup \{b\} \setminus \{a\}))
\end{equation*}

We say that a relation $C(x,y,z)$ \emph{realises the arithmetic on S} if for any $k \leqslant |S|$ there is such $a$, that the relation $B_a(x,y) \rightleftharpoons C(x,y,a)$ realises the number $k$ on $S$. The property to realise the arithmetic can be expressed by the statement:
\begin{multline*}
(\exists z)(\forall i)(s^{B_z}_i = \varnothing) \land\\
\land (\forall z)((B_z \mbox{ \textit{realises a number on S} }) \land (\exists i)(s^{B_z}_i \ne S)\to\\ 
\to (\exists u)(B_u \mbox{ \textit{realises a number on S} }\\
 \land (\exists i,j,a)(a \in S \setminus s^{B_z}_i \land s^{B_{u}}_j=s^{B_z}_i \cup \{a\})))
\end{multline*}

Note that if a relation $C$ realises the arithmetic on $S$, then we can define addition and multiplication on the segment $[0, |S|]$. Addition $S(n,m,l)$ may be defined as
\begin{multline*}
B_n \mbox{ \textit{realises a number on S} } \land B_m \mbox{ \textit{realises a number on S} } \land B_l \mbox{ \textit{realises a number on S} } \land\\
(\exists i,j,k)(s^{B_l}_k=s^{B_n}_i \cup s^{B_m}_j \land s^{B_n}_i \cap s^{B_m}_j = \varnothing)
\end{multline*} 
Multiplication $P(n,m,l)$ may be defined as
\begin{multline*}
B_n \mbox{ \textit{realises a number on S} } \land B_m \mbox{ \textit{realises a number on S} } \land B_l \mbox{ \textit{realises a number on S} } \land\\
(\exists i,j)(s^{B_n}_i \subset s^{B_l}_j \land max(s^{B_l}_j)=max(s^{B_n}_i) \land min(s^{B_l}_j)=min(s^{B_n}_i) \land\\
(\forall a,b \in s^{B_n}_i)(a<b \land (\forall c \in s^{B_n}_i)(a < c \to b \leqslant c) \to (\exists k)(s^{B_m}_k=\{x \in s^{B_l}_j | a \leqslant x <b\})))
\end{multline*} 
(It is not exactly $l=n \cdot m$ but rather $l=n \cdot m+1$ which is not important)

\begin{lemma}\label{model}
If there are definable in $\mathcal {M}$ relations $S( b,x), D( b,x,y,z)$ such that for any natural $n$ for some $ b_n$ the relation $D( b_n,x,y,z)$ realises the arithmetic on $\{x | S( b_n,x)\}$ and $n=|\{x | S( b_n,x)\}|$, then the elementary theory of $\mathcal {M}$ is undecidable.
\end{lemma}
\begin{proof}
Consider an arithmetic formula $(\exists n)Q(n)$ where $Q(x)$ is a bounded quantifiers formula. Under the assumptions of the lemma we can construct the equivalent formula in the structure $\mathcal {M}$, so the elementary theory of $\mathcal {M}$ is undecidable.
\end{proof}

\section{Rank of nodes }

Without loss of generality we suppose that a tree $Tr$ on $\mathbb N$ is a family of finite subsets $\mathbb N$ such that if $s \in Tr$ then any initial segment of $s$ belongs to $Tr$ as well. There is the order $s \preceq s' \rightleftharpoons s$ \emph{is initial segment of} $s'$ on the tree. We say that a relation $Tr(x,y)$ defines the tree, if $\{s_i | s_i = \{x | Tr(x,i)\}\}$ is a tree.

We are going to define the main notion of the article: the rank of a tree node. But before the definition of rank  we need the supporting partial mapping $g\colon Tr \to \mathbb N$. We describe the mapping $g$ in the terms of the $s$-game assigned to a node $s$ of the tree. \textbf{Game}: There are 2 players. In the starting position all items of the node $s$ are drawn on the natural numbers line (red dots):
 \\
\includegraphics[scale=0.5]{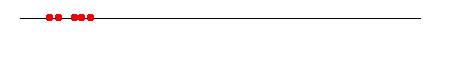}
 \\
 First player mark a boundary $a\geqslant \max(s)$ (black dot).
 \\
\includegraphics[scale=0.5]{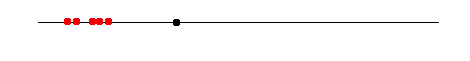}
\newpage
The second player has to choose a finite set $s', \min(s')>a$ of numbers (pink dots)  in such a way, that $s \cup s'$ form new node:
\\  
\includegraphics[scale=0.55]{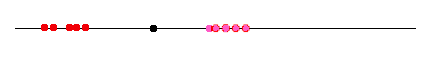}
\\
Now it is first player turn, and so on. 
\\

We set $g(s) = \max k$ [there is a strategy for second player not to lose the game in k moves].
It's easy to note that $g(s) =0 \iff s$ has finite number of sons. We say that a node is \emph{k-regular} if $g(s)=k$ and \emph {regular} if it is $k$-regular for some $k$.

Now we define a \emph{rank} of nodes: a partial mapping $rk \colon Tr \to \mathbb N$: $rk(s)=n \iff (1)\quad any s' \succ s$ is regular and  $ (2)\quad n=\max\{g(s') | s' \succ s\}$.
Note that $rk(s)=0$ if the subtree $\{s' | s \preceq s'\}$ is locally finite. 

We say that a node $s$ of \emph{finite rank} ($rk(s) < \infty$) if $rk(s)$ is defined, otherwise we say that $s$ of \emph{infinite rank} ($rk(s)=\infty$).  

\begin{lemma} \label{rk0}

For any node $s$ 

(i) if $rk(s)=n, s_1 \succ s$, then $s_1$ has finite rank and $rk(s_1) \leqslant rk(s)$.

(ii) if $rk(s)=n, s_1 \succ s$, then $s_1$ has finite rank and $rk(s_1) \leqslant rk(s)$.

(iii) $rk(s)\geqslant g(s)$.

(iv) if $rk(s)=n, s_1 \succ s$, then $s_1$ has finite rank and $rk(s_1) \leqslant rk(s)$.

(v) if $g(s)=n$ then for any $a> \max(s)$ there is a finite set $s', \min(s')>a$ such that $g(s \cup s')=rk(s \cup s')=n-1$.

(vi) if $g(s)=n$ then there is such $a>\max(s)$ that $g(s \cup s')<n$ for any finite subset $s', \min(s')>a, s\cup s'$ is the tree node.

(vii) if $rk(s)=n$ then there are infinitely many pairwise incomparable $s' \succ s$, such that $rk(s')=g(s')=n-1$

\end{lemma}
\begin{proof}
$ $
(i)--(iv) obvious, due to definitions.

(v) for any move  $a \geqslant \max(s)$ of the first player denote by $s'_a$ a best answer of the second player, by the definition of mapping $g$ holds (a) $g(s \cup s'_a) \geqslant n-1$ and (b) there is such $a_0$ (best first move of the first player) such that $g(s \cup s'_a) = n-1$ for all $a \geqslant a_0$. Due to (iv) $rk(s \cup s'_a) \geqslant n-1$.  If $s'_a$ is the best answer, then $g(s \cup s'_a) \geqslant g(s'')$ for all $s'' \succ s \cup s'_a$ so $rk(s \cup s'_a) =g(s \cup s'_a)=n-1$. The existence of infinitely many pairwise incomparable $s_a' \succ s$ for different $a>a_0$ is obvious.

(vi) Denote by $a_0$ a best move of the first player in the $s$-game. Then for any replay $s', \min(s')>a_0$ of the second player holds $g(s \cup s')<n$.

(vii) if $rk(s)=n$, then there is such $s' \succ s$ that $g(s')=n$, so we use (v) here.
\end{proof}

\begin{lemma} \label{fin}
Consider a structure $\mathcal {M} = \langle \mathbb N, \{Tr,<\} \rangle$, where the relation $Tr( a,x,y)$ with a parameter $ a$ defines a family of trees on $\mathbb N$ and $<$ is the usual order on $\mathbb N$. If the elementary theory of $\mathcal {M}$ is decidable, then there is such number $k$, that $rk(s)<k$ holds for all nodes $s$ of finite rank in all trees $Tr( a,x,y)$.
\end{lemma}
\begin{proof}
In the contrary: we suppose that there are nodes of arbitrary big finite rank and show that conditions of lemma \ref{model} hold. We fix a value of the parameter $ a_0$ and consider the tree $Tr=Tr( a_0,x,y)$.

We define functions $\varphi(x), \psi(x,y)$ on $\mathbb N$ in the following way: for a number $a \in \mathbb N$ consider the segment $[0,a]$ and choose a node $s \subset [0,a]$, let $g(s)=k$-regular for some $k$. Then (Lemma \ref{rk0} (vi)) exists $b_s>a$, such that $(s' \succ s, s' \cap (max(s)+1,b_s)=\varnothing) \Rightarrow g(s') \leqslant k-1$. Define $\varphi(a)$ such that $\varphi(a) > b_s$ for all regular $s \subset [0,a]$.

Let $a<b$. Choose a node $s \subset [0,a]$, let $s$ be $k$-regular for some $k$. Then (Lemma \ref{rk0} (v)) exists $s' \succ s, s' \cap (max(s),b)=\varnothing, g(s')=k-1$ . Define $\psi(a,b)$ such that $\psi(a,b)>\max(s')$ for all regular $s \subset [0,a]$.

Note that functions $\varphi, \psi$ are monotonic. We do not assert (yet) that they are definable in $\mathcal M$.

To continue the proof of lemma \ref{fin} we need two following lemmas:

\begin{lemma}\label{tmp1} Suppose that for $a_1<b_1<\dots<a_n<b_n<a_{n+1}$ holds $b_i>\varphi(a_i), a_{i+1}>\psi(a_i,b_i)$ . Choose $s \subset [0,a_1], g(s)=k<n$. Then

(i) for any $u \subset [1,n-1], |u|=k$ there is such node $s' \succ s$, that $s' \cap (a_i,b_i)=\varnothing$ for all $i<n$ and $\{i | s' \cap [b_i, a_{i+1}] \ne \varnothing\}=u$.

(ii) if $s' \succ s$ and $s' \cap (a_i,b_i)=\varnothing$ for all $i \leqslant n$, then $|\{i | s' \cap [b_i, a_{i+1}] \ne \varnothing\}| \leqslant k$.
\end{lemma}
\begin{proof}
$ $

(i) induction on $k$. Let $i=\min(u)$. By definition of  mapping $\psi$ and because $s \subset [0, a_i]$,  there is such $s' \succ s$, that $s' \cap (max(s)+1,b_i)=\varnothing, max (s')<a_{i+1}, g(s')=k-1$. So we can apply an inductive hypothesis to the collection $a_{i+1}<b_{i+1}<\dots<b_n<a_{n+1}$, the node $s'$ an the set $u \setminus \{i\}$.

(ii) suppose that $s' \succ s$ and $s' \cap (a_i,b_i)=\varnothing$ for all $i<n$ and $\{i | s' \cap [b_i, a_{i+1}] \ne \varnothing\}=\{c_1<c_2<\dots<c_m\}$. By induction on $i$ show that $r(s' \cap [0,a_{{c_i}+1}]) \leqslant k-i$. If $r(s' \cap [0,a_{c_{i-1}+1}])=m$, then by definition of mapping $\varphi$ for any $s'' \succ s', s'' \cap [a_{c_i},b_{c_i}]=\varnothing$ holds $r(s'')<rk(s'')\leqslant m$, i.e. $r(s' \cap [0,a_{c_{i+1}}]) < r(s' \cap [0,a_{c_i}])$.
\end{proof}

\begin{lemma}\label{tmp2}
For any $u, v \succ s$ we denote $A_{u,v} \rightleftharpoons \{max(s)\} \cup \{a \in u \setminus s |  (\forall a' < a, a' \in u )([a',a] \cap v \ne \varnothing)\}, B_{u,v} \rightleftharpoons \{b \in v \setminus s |  (\forall b' < b)([b',b] \cap u \ne \varnothing)\}$. If $r(s)=n+1$ then there are such $u, v \succ s$ that  sets $A_{u,v}=\{a_1<\dots<a_{n+1}\}$ and $B_{u,v}=\{b_1<\dots<b_n\}$ meet the conditions of lemma \ref{tmp1}
\end{lemma}
\begin{proof}
We will construct collections $u_0 \prec u_1 \prec \dots \prec u_n$=$v_0 \prec v_1 \prec \dots \prec v_{n-1}=v$ such that $u_0=v_0=s, r(u_i)=r(v_i)=n+1-i$. Suppose that  $u_i, v_i$  are already constructed, $max(v_i) \leqslant max(u_i)$. Choose such $v_{i+1} \succ v_i$ that $g(v_{i+1})=g(v_i)-1, v_{i+1} \cap (max(v_i)+1,\varphi(max({u_i}) = \varnothing$. Since $min(v_{i+1}\setminus v_i) > u_i$, so $A_{u_i,v_{i+1}}=A_{u_i,v_{i}}, B_{u_i,v_{i+1}}=B_{u_i,v_{i}} \cup \{min(v_{i+1})\setminus v_i)\}$ and $min(v_{i+1}\setminus v_i) > \varphi(max(u_i)) \geqslant \varphi(max(A_{u_i,v_{i}}))$.

Now we in the same way choose the node $u_{i+1}$ considering the node $v_{i+1}$ instead of $u_i$ and the number $\psi(max(u_i), max(v_{i+1}))$ instead of $\varphi(max(u_i))$.
\end{proof}

Continue the proof of lemma \ref{fin}. Suppose that there exist nodes of arbitrary big finite rank. Fix some $n \in  \mathbb N$. According the lemma  \ref{tmp2} there are  such $u,v$, that the sets $A_{u,v}, B_{u,v}$ meet the conditions of lemma \ref{tmp1} and $|A_{u,v}|=n$. Since we can choose $A_{u,v}$ such that minimal member $a_1$ of this set is arbitrary big, we suppose that  $a_1>s_1, a_1>s_2, \dots, a_1>s_n$ for some nodes $g(s_i)=i$. We will interpret the arithmetic of segment  $[0,n]$ on $A_{u,v}$, the node $s_i$ will realise the number $i$. Namely for any node $s$ we define the finite subset $A_s \subset A_{u,v}$ so that  $A_s=\{a_i | (\forall k)((a_k, b_k) \cap s=\varnothing) \land [b_{i-1},a_i] \cap s \ne \varnothing \}$. It is obvious that there is a simple formula $Q(u,v,s,x)$ in the structure $\mathcal M$ defining the relation $x \in A_s$. For any $s_i$ we consider all $s \succ s_i$ such that the subset  $A_s$ is maximal. According the lemma \ref{tmp1} they will be all $i$-element subsets of $A_{u,v}$ and so the $s_i$ realises the number $i$ on  $A_{u,v}$. According the lemma \ref{model} the elementary theory of the structure  $\mathcal M$ is undecidable.
\end{proof}

\begin{seq}\label{seq1}
Let a relation $Tr(y,x)$ defines a tree on $\mathbb N$, and elementary theory of the structure $\mathcal M=\langle \mathbb N, \{Tr,<\} \rangle$ is decidable. Then

(i) the relation "$s$ \emph{is a node of finite rank}" is definable (in $\mathcal M$).

(ii) the functions $\varphi, \psi$ are definable.

(iii) if the set of nodes of infinite rank is not empty, then it contains a definable subtree isomorphic to $\mathbb N^{<\omega}$.

(iv) there is $k \in \mathbb N$ such that for any node $s$ of finite rank, $s=\{a_1<a_2<\dots<a_n\}$, holds $k \geqslant |\{a_i \in s | a_{i+1} > \varphi(a_i)\}|$.
\end{seq}
\begin{proof}
According the lemma \ref{fin} there is such $k \in \mathbb N$ that $k \geqslant rk(s)$ for every node $s$ of finite rank.

(i) the relation "$rk(s)=0$" is definable, so by induction the relation "$rk(s)=i$" is definable for any $i$. Then the relation "$s$ \emph{is a node of finite rank}" is equivalent to $\bigvee_{i=0}^{k}rk(s)=i$.

(ii) immediately  follows from (i).

(iii) the relation $inf(s) \rightleftharpoons$ "$s$ \emph{is a node of infinite rank}" is definable. To proof the isomorphism to $\mathbb N^{<\omega}$ it is enough to show, that for any node $s$  of infinite rank there is a node of  infinite rank $s' \succeq s$, such that $(\forall a)(\exists s'' \succ s')(inf(s'') \land s'' \cap [max(s'')+1,a]=\varnothing)$. On contrary suppose that $(\forall s' \succeq s)(\exists a)(\forall s'' \succ s')((s'' \cap [max(s'')+1,a]=\varnothing) \to rk(s'') \leqslant k)$, then, by definition the node $s$ has finite rank.

(iv) from the definition of the function $\varphi$ follows that $(a_{i+1} > \varphi(a_i)) \Rightarrow rk(\{a_1,\dots,a_{i+1}\}) < rk(\{a_1,\dots,a_{i}\})$.
\end{proof}
\begin{seq}\label{seq2}
Let a relation $Tr(a,y,x)$ with the parameter $a$ defines a family of trees on $\mathbb N$, elementary theory of the structure $\mathcal M=\langle \mathbb N, \{Tr,<\} \rangle$ is decidable. Then

(i) the relation $Q(a) \rightleftharpoons$ "\emph{there is an infinite branch in the tree}  $Tr(a,y,x)$" is definable.

(ii) if there is an infinite branch in the tree  $Tr(a,y,x)$ then there is a definable infinite branch.
\end{seq}
\begin{proof}
According the lemma \ref{fin} there is such number $k$ that  $k \geqslant rk(s)$ holds for all nodes $s$ of finite rank in all trees $Tr(a,y,x)$. So the relation  $inf(a,s) \rightleftharpoons$ "$s$ \emph{is a node of infinite rank in the tree} $Tr(a,y,x)$" is definable. Consider two cases.

(1) There is a node of infinite rank in the tree $Tr(a,y,x)$. Then due to  sequence \ref{seq1}(ii) there is a definable infinite branch in the tree.

(2) All nodes of the tree $Tr(a,y,x)$ are of finite rank. To any node $s \in Tr(a)$ asssign the subtree $Tr_s=\{s' \succ s | rk(s')=rk(s)\}$. We show that there is an infinite branch in the tree $Tr(a,y,x)$ if and only if the tree $Tr_s$ is infinite for some $s$. Note that the tree $Tr_s$ is locally finite. Indeed, if a node $s$ has an infinitely many sons  $s' \succ s, rk(s')=rk(s)$, then $(\forall a)(\exists s' \succ s)(s' \cap [max(s')+1,a]=\varnothing \land rk(s')=rk(s))$, and, by the definition of the function $r$, holds $rk(s) \leqslant r(s) < rk(s)$. So if the tree $Tr_s$ is infinite, then there is a definable infinite branch, which is the branch in the tree $Tr(a,y,x)$ as well.

Conversely, suppose that  in the tree $Tr(a,y,x)$ exists an infinite branch $s_1 \prec \dots \prec s_n \prec \dots$. Because $rk(s_i) \geqslant rk(s_{i+1})$, so for some $n$ and for all $i>0$ holds $rk(s_n)=rk(s_{n+i})$ and the tree $Tr_{s_n}$ is infinite.

\end{proof}


\end{document}